\documentclass[11pt]{amsart}
\usepackage[margin=1.1in]{geometry}
\usepackage{amsmath,amssymb,amsthm,mathtools}
\usepackage[T1]{fontenc}
\usepackage{lmodern}
\usepackage{microtype}
\usepackage[colorlinks=true,linkcolor=blue,citecolor=blue,urlcolor=blue]{hyperref}

\newtheorem{theorem}{Theorem}
\newtheorem{lemma}[theorem]{Lemma}
\newtheorem{proposition}[theorem]{Proposition}

\newcommand{\tr}{\operatorname{tr}}
\newcommand{\Cl}{\mathcal C}

\title{More than $83.69\%$ of the zeros of the Riemann zeta function are distinct}
\author{Kristian Muri Knausg{\aa}rd}
\email{kristianmk@ieee.org}
\date{26 September 2026}

\begin{document}

\begin{abstract}
The lower asymptotic proportion of distinct nontrivial zeros of the Riemann
zeta function, relative to the total number counted with multiplicity, is at
least $0.8369928814\ldots$.
Earlier work proves $0.83625\ldots$, and a report we have not verified claims
$0.83672\ldots$. As in the proof of the bound $0.83625$, an
unconditional version of Montgomery's pair-correlation theorem gives an
asymptotic energy estimate. The new ingredient is a short matrix inequality
with a free clipping parameter. It strengthens the lower bound for this energy
in terms of the number of distinct zeros. The gain is a nonnegative correction
from overlaps between different nearby zeros on the critical line, which is
retained even
when some of these zeros are double. The matrix
inequality, the threshold lemma, the block dichotomy, the counting assembly
and the exact arithmetic are proved formally in Lean~4. The constant relies on
a computer-assisted local inequality from recent work that has not yet been
refereed. That computation was re-run independently, and every imported input
is listed. This paper is primarily an experiment in AI-assisted mathematical
research (Section~\ref{sec:verify}).
\end{abstract}

\maketitle

\section{Introduction}

The Riemann zeta function has infinitely many nontrivial zeros
$\rho=\sigma+i\gamma$ with $0<\sigma<1$. It is conjectured that every zero is
\emph{simple}, meaning not repeated. Let $N(T)$ count the zeros with
$0<\gamma\le T$, with multiplicity, and let $N_d(T)$ count the \emph{distinct}
ones. If all zeros were simple, then $N_d=N$.

Montgomery's pair-correlation method~\cite{Mon73} gives lower bounds for such
proportions. An argument discovered by Claude (Anthropic)~\cite{Claude26} and
verified by Alp\"oge and Furman~\cite{AF}, together with a second proof by
Lamzouri~\cite{Lam}, shows unconditionally that
$\liminf N_d(T)/N(T)\ge0.8362503518\ldots$. Wang~\cite{Wang} improved this
by about $3\cdot10^{-8}$. Chandru~\cite{Chandru} obtained $0.83652$ with a
computer-assisted inequality for seven points, in a preprint that has not been
refereed; the reports~\cite{Schaefer,Singh} state $0.83671$ and $0.83672$. A
report~\cite{YY} claims a much larger value, $0.8981$, from trace moments of
order up to six; its authors grade it as a candidate whose analytic part is not
yet verified, and we do not count it among the values above. We have not
verified these reports, and no result of this paper depends on them. The
purpose of this paper is to prove the following bound.

\begin{theorem}\label{thm:main}
$\displaystyle\liminf_{T\to\infty}\frac{N_d(T)}{N(T)}\ \ge\
q:=\frac{16260119298029}{19426831050000}=0.83699288145242\ldots$
\end{theorem}

\emph{The idea.} Rescale the zeros so that their mean spacing is one: a zero
$\rho=\frac12+i\gamma$ on the critical line gets the coordinate
$y_\rho=\gamma\log T/(2\pi)$. Fix an even probability density $f=\eta^2$ on
$[-\frac12,\frac12]$, and attach to this zero the wave packet
\[
 F_{y_\rho}(u)=\eta(u)\,e^{2\pi iuy_\rho}\in L^2\bigl(-\tfrac12,\tfrac12\bigr).
\]
These are unit vectors, and $\langle F_y,F_{y'}\rangle=\hat f(y-y')$, where
$\hat f(x)=\int f(u)e^{-2\pi ixu}\,du$. The overlap is $1$ for coincident zeros,
close to $1$ for nearby zeros, and small for zeros far apart.
Section~\ref{sec:count} uses a smoothed version of these vectors, labelled there
by $z_\rho=-y_\rho$.

Zeros off the critical line come in pairs placed symmetrically about the line.
Each pair contributes a Hermitian term with at most one positive eigenvalue, so
the construction needs no assumption about where the zeros lie. Summing the
projections onto the vectors above and these pair terms, with multiplicity,
gives an operator $A$ with $\tr A=N$. The energy $E=\tr A^2$ is a pair sum over
all zeros. For ordered pairs $(\rho,\rho')$ on the critical line, including
$\rho=\rho'$, the summands are $\hat f(y_\rho-y_{\rho'})^2$, counted with
multiplicity. The full identity, including off-line zeros, is given in
Section~\ref{sec:count}. Montgomery's pair-correlation theorem, in its unconditional
form~\cite{BGST}, computes this energy asymptotically: for the smoothed operator,
$E=(C_\varepsilon+o(1))N$ with
$C_\varepsilon\to C=\int f^2+\iint|s-t|f(s)f(t)\,ds\,dt$, and $C\le2-H<2$ for the
window used here (Proposition~\ref{prop:energy}).

A zero on the critical line of multiplicity $\nu$ contributes $\nu$ to the trace
and $\nu^2$ to the diagonal part of the energy sum. A double zero contributes $4$
to this diagonal part, whereas two distinct simple zeros contribute $2$, with an
additional mutual term $2\hat f(y-y')^2$ that is small when they are far apart.
Both configurations contribute $2$ to the trace. Linear algebra turns this into $E\ge3N-2N_d$,
hence $\liminf_{T\to\infty}N_d/N\ge\frac{3-C}2$. The Montgomery--Taylor window
gives the bound $0.83625$ this way. The window used here gives up about
$2\cdot10^{-5}$ of that baseline in exchange for a larger gain below.

The improvement adds a nonnegative \emph{defect} term to this inequality. The
defect measures how far the Gram matrix $U=(\hat f(y_i-y_j))$ of the distinct
on-line zeros of multiplicity one or two is from the identity, and nearby zeros
make it positive. The total spread of these points is at most $N+o(N)$. A
certified local inequality (Proposition~\ref{prop:seven}) relates the spread to
the squared overlaps: for any seven consecutive points, a weighted sum of their squared overlaps plus a multiple of
their spread is at least a fixed positive constant. As a result, the defect is
at least a fixed amount per point, minus a penalty proportional to how widely
the points are spread.

Recent reports on simple zeros~\cite{Ainta,Trmdy,Shi,Tawan}, which have not
been refereed, use such a defect for simple zeros. Counting distinct zeros requires keeping the defect when double
zeros are present too, and Theorem~\ref{thm:gram} does exactly that.

\section{The matrix inequalities}\label{sec:matrix}

For Hermitian $X$ and real $f$ write $\tr f(X)=\sum_i f(\lambda_i(X))$. Write
$X\preceq Y$ when $Y-X$ is positive semidefinite. For real $c$ and $\tau\ge0$ let
\[
 \varphi_c(x)=\begin{cases}x^2,&x\le c,\\2cx-c^2,&x>c,\end{cases}
 \qquad
 \Psi_\tau(x)=\begin{cases}(x-1)^2,&x\le 1+\tau,\\2\tau(x-1)-\tau^2,&x>1+\tau.\end{cases}
\]
Both functions are $C^1$ with nondecreasing derivative, hence convex on
$\mathbb R$.

\begin{theorem}[Mixed-multiplicity Gram inequality]\label{thm:gram}
Let $U$ be a Hermitian $l\times l$ matrix with unit diagonal, let
$D=\operatorname{diag}(d_i)$ with $1\le d_i\le2$, and let $G=D^{1/2}UD^{1/2}$.
If $\tau\ge0$ and $c\ge\max(1+\tau,\,2+\tau/2)$, then
\[
 \tr\varphi_c(G)-\tr D^2\ \ge\ \tr\Psi_\tau(U).
\]
\end{theorem}

\begin{lemma}\label{lem:var}
If $G,B$ are Hermitian and $B\preceq2cI$, then
$\tr\varphi_c(G)\ge\tr(BG)-\frac14\tr B^2$.
\end{lemma}

In control-theoretic terms, Lemma~\ref{lem:var} is a dual certificate. Every
matrix $B$ that satisfies the linear matrix inequality $B\preceq2cI$ yields
a lower bound on $\tr\varphi_c(G)$, and the bound is attained at
$B=2\min(G,cI)$. The proof of Theorem~\ref{thm:gram} exhibits one admissible $B$.

\begin{proof}[Proof of Lemma~\ref{lem:var}]
In an eigenbasis of $G$, $\tr(BG)=\sum b_{ii}\lambda_i$ and
$\tr B^2\ge\sum b_{ii}^2$, with $b_{ii}\le2c$. It suffices to show
$b\lambda-\frac14b^2=\lambda^2-(\lambda-\frac b2)^2\le\varphi_c(\lambda)$
for $b\le2c$. If $\lambda\le c$ this is clear. If $\lambda>c$, then
$\lambda-\frac b2\ge\lambda-c>0$, and the claim follows.
\end{proof}

\begin{proof}[Proof of Theorem~\ref{thm:gram}]
Let $\Cl$ be $U-I$ with each eigenvalue $\mu$ replaced by $\min(\mu,\tau)$, so
that $\Cl\preceq\tau I$. Put $M=D^{-1/2}\Cl D^{-1/2}$ and $B=2D+2M$. Then
$B\preceq2D+2\tau D^{-1}\preceq2cI$. For the last step, $d+\tau/d$ is convex
on $[1,2]$, so it is at most $\max(1+\tau,2+\tau/2)\le c$. Using $U_{ii}=1$ and
cyclicity, $\tr(DG)=\tr D^2$, $\tr(MG)=\tr(\Cl U)$ and $\tr(DM)=\tr \Cl$. Hence
\[
 \tr(BG)-\tfrac14\tr B^2-\tr D^2=2\tr\bigl(\Cl(U-I)\bigr)-\tr M^2 .
\]
Moreover $\tr M^2=\sum|\Cl_{ij}|^2/(d_id_j)\le\tr \Cl^2$. By Lemma~\ref{lem:var},
\[
 \tr\varphi_c(G)-\tr D^2\ \ge\ 2\tr\bigl(\Cl(U-I)\bigr)-\tr \Cl^2
 =\sum_\mu\bigl(2\mu\min(\mu,\tau)-\min(\mu,\tau)^2\bigr),
\]
and each summand equals $\Psi_\tau(1+\mu)$.
\end{proof}

The counting argument also requires a threshold inequality of the kind used
in~\cite{Claude26,AF}. The following proof avoids eigenvalue interlacing. For
Hermitian $X$ let $X_\pm\succeq0$ carry its positive and negative eigenvalues,
so that $X=X_+-X_-$ and $X_+X_-=0$.

\begin{lemma}\label{lem:threshold}
Let $P\succeq0$, let $Q$ be Hermitian with at most $b$ positive eigenvalues,
let $A=P+Q$, and let $c$ be real. Then
$\tr A^2\ge2c\tr A-2c\tr P+\tr\varphi_c(P)-c^2b$.
\end{lemma}

\begin{proof}
Write $Z=Q_-$, so that $A=(P-Z)+Q_+$. The cross term is
$2\tr((P-Z)Q_+)=2\tr(PQ_+)\ge0$, since $ZQ_+=0$. Hence
$\tr A^2\ge\tr(P-Z)^2+\tr Q_+^2$. Also $\tr Q_+^2\ge2c\tr Q_+-c^2b$, because
$x^2\ge2cx-c^2$. It remains to show $\tr(P-Z)^2+2c\tr Z\ge\tr\varphi_c(P)$.
Since $\varphi_c(p)=p^2-(p-c)_+^2$, this is
$\tr X_+^2-2\tr(XZ)+\tr Z^2\ge0$ with $X=P-cI$. Now
$\tr(XZ)\le\tr(X_+Z)$ because $\tr(X_-Z)\ge0$. So the left side is at least
$\tr(X_+-Z)^2\ge0$.
\end{proof}

\section{Proof of Theorem~\ref{thm:main}}\label{sec:count}

\emph{The window.} Following~\cite{Trmdy} (files \texttt{design.py} and
\texttt{kernel.py}), put
\[
 v(u)=\cos(\sqrt2\,u)+10^{-9}\sum_{j=1}^{6}\kappa_j\cos(2\pi ju)\quad(|u|\le\tfrac12),
 \qquad f=v\Big/\!\int_{-1/2}^{1/2}v ,
\]
with $(\kappa_j)=(3322500,-7609135,1190194,-731476,-1680572,1141360)$.
Both $v$ and $f$ vanish outside $I=[-\frac12,\frac12]$. On $I$, $v>0$ (its
minimum there is about $0.75$), and $\int v=\sqrt2\sin(1/\sqrt2)$. Let
$K(x)=\hat f(x)=\int f(u)e^{-2\pi ixu}\,du$, so that $K(0)=1$ and
$K^2=(\hat v/\hat v(0))^2$ is the verifier's kernel.

\emph{The operator.} Let $L=\log T$ and $z_\rho=i(\rho-\frac12)L/(2\pi)$ for
$0<\gamma\le T$. On-line zeros give real $z_\rho$. An off-line zero pairs with
$1-\bar\rho$, which has the same multiplicity and the point $\bar z_\rho$.

The function $f$ jumps at $\pm\frac12$, so $f\ast f$ is not twice
differentiable, and the smoothing below is necessary. Since $\sqrt f$ is smooth
and positive on the interior of $I$, multiplying it by even smooth cutoffs
that tend to $\mathbf 1_I$ and renormalising gives
$f_\varepsilon=\eta_\varepsilon^2$. Here $\eta_\varepsilon$ is real, even,
smooth and supported inside $I$, $\int f_\varepsilon=1$,
$d_\varepsilon=\|f_\varepsilon-f\|_1\to0$ and $f_\varepsilon\to f$ in $L^2$.

Put $K_\varepsilon=\hat f_\varepsilon$ and
$F_z(u)=\eta_\varepsilon(u)e^{-2\pi iuz}$. With inner products conjugate-linear
in the first slot, $\langle F_z,F_w\rangle=K_\varepsilon(w-\bar z)$. The operator
$A=\sum_z\nu_zF_z\langle F_{\bar z},\cdot\rangle$, where $\nu_z$ is the
multiplicity, is Hermitian. It satisfies $\tr A=N$ and
$\tr A^2=\sum_{z,w}\nu_z\nu_wK_\varepsilon(z-w)^2=:E_\varepsilon$, using that
$K_\varepsilon$ is even. All operators act on the finite-dimensional span of
the $F_z$.

\begin{proposition}[{analytic input \cite[Lemma~5]{BGST}; computer-assisted bound \cite{Trmdy}}]\label{prop:energy}
$E_\varepsilon=(C(f_\varepsilon)+o(1))N$, where
$C(g)=\int g^2+\iint|s-t|g(s)g(t)$. Moreover
$C(f_\varepsilon)\to C(f)\le2-H$ with $H=\frac{3362285207}{5000000000}$.
\end{proposition}

The asymptotic is the unconditional Montgomery theorem~\cite[Lemma~5]{BGST}
(see also~\cite{Lam,AF}). It applies to real, even, integrable test functions
supported in $[-1,1]$ and Lipschitz at $0$, and it sums over all zeros,
on or off the line, with multiplicity. It carries the weight
$w(\rho-\rho')=4/(4-(\rho-\rho')^2)$. With $R_\varepsilon=f_\varepsilon*f_\varepsilon$,
\[
 \bigl(R_\varepsilon-R_\varepsilon''/(4L^2)\bigr)^{\wedge}(x)
 =\Bigl(1+\frac{\pi^2x^2}{L^2}\Bigr)K_\varepsilon(x)^2 .
\]
At $x=z_\rho-z_{\rho'}$ the factor is $1-(\rho-\rho')^2/4=1/w(\rho-\rho')$.
Both $R_\varepsilon$ and $R_\varepsilon''$ are admissible fixed test functions.
So applying the lemma to both removes the weight exactly. The
$R_\varepsilon''$ term contributes $O_\varepsilon(N/L^2)$, and the total error
is $o_\varepsilon(N)$.

The bound $C(f)\le2-H$ is a certified interval enclosure~\cite{Trmdy}. An
independent 40-digit (non-interval) evaluation gives
$2-C(f)=0.67245704141454428878\ldots$, consistent with it.

\emph{Splitting.} Let $n_1,n_2$ count the distinct on-line zeros of
multiplicity $1,2$. Let $h$ count those of multiplicity $\ge3$, and let $k$
count off-line pairs. Then $N_d=n_1+n_2+h+2k$. Let $P$ be the part of $A$
coming from the $l=n_1+n_2$ low-multiplicity points $y_j$. Then $P=VDV^*$ with
$Ve_j=F_{y_j}$, and $U=V^*V$ has $U_{ij}=K_\varepsilon(y_i-y_j)$. So $P$ has
the nonzero spectrum of $G=D^{1/2}UD^{1/2}$. Since $\varphi_c(0)=0$, it follows
that $\tr\varphi_c(P)=\tr\varphi_c(G)$. Also $\tr P=n_1+2n_2$.

In $Q=A-P$, each high point gives a positive rank-one term. Each off-line pair
gives $\nu(vw^*+wv^*)=\frac\nu2\bigl((v+w)(v+w)^*-(v-w)(v-w)^*\bigr)$, with
$v=F_z$ and $w=F_{\bar z}$, which has at most one positive eigenvalue. Since
$n_+(\sum_jQ_j)\le\sum_jn_+(Q_j)$, $Q$ has at most $h+k$ positive eigenvalues.

Fix $\tau\ge0$ and $c\ge\max(1+\tau,2+\tau/2)$, so that $c\ge2$.
Lemma~\ref{lem:threshold} with $b=h+k$ and Theorem~\ref{thm:gram} give
\[
 E_\varepsilon\ \ge\ 2cN-(2c-1)n_1-(4c-4)n_2-c^2(h+k)+\tr\Psi_\tau(U).
\]
Compare this with $3N-2N_d$. A high point of multiplicity $j\ge3$ leaves
$(2c-3)j+2-c^2\ge r_h:=6c-7-c^2$, and an off-line pair of multiplicity $j\ge1$
leaves $(4c-6)j+4-c^2\ge r_k:=4c-2-c^2$. Hence
\begin{equation}\label{eq:key}
 E_\varepsilon\ \ge\ 3N-2N_d+\tr\Psi_\tau(U)+r_hh+r_kk.
\end{equation}

\emph{The local defect.}

\begin{proposition}[imported, computer-assisted~\cite{Trmdy}]\label{prop:seven}
There are rational $w_{ij}\ge0$ with $\sum_iw_{i,i+r}=2$ for $r=1,\dots,6$,
such that for all real $y_0\le\dots\le y_6$,
\[
 \tfrac1{2736}(y_6-y_0)+\textstyle\sum_{0\le i<j\le6}w_{ij}K(y_j-y_i)^2\ \ge\ \delta:=\tfrac{891}{200000}.
\]
\end{proposition}

Take the Gram matrix $U_B$, built with $K$, of $m$ consecutive points. It is
positive semidefinite because $f\ge0$. Summing Proposition~\ref{prop:seven} over
the $m-6$ windows gives weight at most $2$ per pair. Since
$\tr(U_B-I)^2=2\sum_{i<j}K(y_j-y_i)^2$, this gives
$\tr(U_B-I)^2+\frac1{2736}W\ge\Delta_m:=(m-6)\delta$, where $W$ is the total
window length. If $\Delta_m\le\tau^2$, then
$\tr\Psi_\tau(U_B)+\frac1{2736}W\ge\Delta_m$. Indeed, if every eigenvalue is at
most $1+\tau$, then $\tr\Psi_\tau(U_B)=\tr(U_B-I)^2$. Otherwise one eigenvalue
alone contributes at least $\tau^2$, and $\Psi_\tau\ge0$.

Replacing $K$ by $K_\varepsilon$ changes each entry by at most $d_\varepsilon$,
so the block changes by at most $md_\varepsilon$ in norm. Both blocks are Gram
matrices, so their eigenvalues lie in $[0,\infty)$, where $\Psi_\tau$ is
$\max(2,2\tau)$-Lipschitz. By Weyl's inequality, $\tr\Psi_\tau$ changes by at
most $\max(2,2\tau)m^2d_\varepsilon$.

By convexity of $\Psi_\tau$ (the pinching inequality
$\tr\Psi_\tau(X)\ge\sum_B\tr\Psi_\tau(X_{BB})$), $\tr\Psi_\tau(U)$ dominates
the sum over any partition into consecutive blocks. Average over the $m$
shifted partitions:
\begin{itemize}
\item the full blocks over all shifts number at least $l-2m$, so the
 $\Delta_m$ terms total at least $\Delta_m(l-2m)/m$;
\item each window lies inside a block for at most $m-6$ of the shifts, and each
 gap lies in at most $6$ windows, so the $W$ terms total at most
 $6(m-6)(y_l-y_1)/(2736m)$.
\end{itemize}
The result is
\begin{equation}\label{eq:defect}
 \tr\Psi_\tau(U)\ \ge\ a\,l-\beta\,(y_l-y_1)-O_{m,\tau}(d_\varepsilon)\,N-O_m(1),
 \qquad a=\frac{\Delta_m}m,\quad \beta=\frac{6(m-6)}{2736\,m}.
\end{equation}
Here the spread $y_l-y_1$ is defined to be $0$ when $l\le1$. If no shift has
a full block, \eqref{eq:defect} still holds: its right side is then at most
$0$ once the constant in $O_m(1)$ is taken at least $2\Delta_m$, and
$\tr\Psi_\tau(U)\ge0$.
Shifted-block pinching is from~\cite{Ainta}. The window-in-block pressure count
is from~\cite{Tawan}; see also~\cite{Trmdy,Shi}. What is new is the use of both
on \emph{all} distinct points of multiplicity one or two.

\emph{Constants.} Take $\tau=\frac{12}5$, $c=\frac{17}5$ and $m=1298$. Then
$\Delta_m=\frac{287793}{50000}<\tau^2=\frac{144}{25}$, $a=\frac{26163}{5900000}$ and
$\beta=\frac{17}{7788}$. Also $r_h-a=\frac{10829837}{5900000}>0$ and
$r_k-2a=\frac{91837}{2950000}>0$.

\emph{Conclusion.} The points $y_j$ lie in an interval of length
$TL/2\pi=N+O(T)=N+o(N)$, and $l=N_d-h-2k$. So~\eqref{eq:key}, \eqref{eq:defect}
and Proposition~\ref{prop:energy} give
\[
 (2-a)N_d\ \ge\ \bigl(3-C(f_\varepsilon)-\beta-O_{m,\tau}(d_\varepsilon)\bigr)N
 +(r_h-a)h+(r_k-2a)k-o(N).
\]
Let $T\to\infty$ and then $\varepsilon\to0$. This gives
$\liminf N_d/N\ge(1+H-\beta)/(2-a)=q$ in exact arithmetic. \qed

\section{Verification and nature of this work}\label{sec:verify}

\emph{Formal proofs.} The following are proved in Lean~4 with
Mathlib~\cite{mathlib}; the source is in the ancillary files:
\begin{itemize}
\item Lemma~\ref{lem:var};
\item Theorem~\ref{thm:gram}, for all $\tau\ge0$ and $c\ge\max(1+\tau,2+\tau/2)$;
\item Lemma~\ref{lem:threshold}, in both exact-count and ``at most $b$'' forms;
\item the block dichotomy;
\item the combination of Lemma~\ref{lem:threshold} and Theorem~\ref{thm:gram}
 (the inequality preceding~\eqref{eq:key});
\item the counting assembly leading to $(2-a)N_d\ge(1+H-\beta)N-{}$errors, from
 explicitly stated hypotheses;
\item all exact constants.
\end{itemize}
Every theorem depends only on the axioms \texttt{propext},
\texttt{Classical.choice} and \texttt{Quot.sound}. The code contains no
\texttt{sorry}.

Not formalized:
\begin{itemize}
\item the operator construction and its reduction to matrices: $\tr P=n_1+2n_2$,
 $\tr\varphi_c(P)=\tr\varphi_c(G)$, and $n_+(Q)\le h+k$;
\item the energy asymptotic;
\item the certified local inequality;
\item the pinching and block averaging;
\item the smoothing $K\to K_\varepsilon$.
\end{itemize}
These enter the Lean statements only as explicit hypotheses.

\emph{Computation.} Proposition~\ref{prop:seven} comes from the
Arb-based~\cite{Arb} verifier of~\cite{Trmdy}, commit \texttt{1610b97}. Its
exhaustive search was re-run on separate hardware: 96 shards and
$2{,}168{,}370$ interval nodes, all verified. Both table hashes and all pruning
counts match the published run. The run record is in the ancillary files. The
verifier code was not audited line by line. As noted upstream, this
certificate has no tangent-free replay yet: $453{,}136$ of its nodes are
certified by a convex-tangent pruner. The enclosure of $C(f)$ is taken
from~\cite{Trmdy}, with the independent check stated above.

\emph{Status.} The analytic inputs~\cite{AF,Lam,BGST} and the
certificates~\cite{Ainta,Trmdy,Shi,Tawan} are recent and mostly not yet refereed.

\emph{Nature of this work.} This paper is primarily an experiment in
AI-assisted mathematical research. The arguments, the Lean formalization and
much of the text were developed with substantial assistance from AI systems
(OpenAI Codex and Anthropic Claude) under the author's direction. The author
defined and steered the research programme, reviewed all steps, and found the
$\tau$-generalisation together with OpenAI Codex. The author takes full
responsibility for the content.

\bigskip

\end{document}